\documentclass[reqno]{amsart}
\usepackage{latexsym}
\usepackage{amssymb,amsthm,amsmath}
\usepackage[dvips]{graphicx}

\usepackage{color}
\usepackage[colorlinks=true]{hyperref}
\hypersetup{urlcolor=blue, citecolor=blue}
\usepackage{tikz}
\usepackage{multicol}

\theoremstyle{plain}
\newtheorem{theorem}{Theorem}
\newtheorem{lemma}[theorem]{Lemma}
\newtheorem{corollary}[theorem]{Corollary}
\newtheorem{proposition}[theorem]{Proposition}

\theoremstyle{definition}
\newtheorem{definition}[theorem]{Definition}
\theoremstyle{remark}
\newtheorem{remark}[theorem]{Remark}

\def\d#1{{#1\kern-0.4em\char"16\kern-0.1em}}
\def\D#1{{\raise0.2ex\hbox{-}\kern-0.4em #1}}
\newcounter{zd}

\newcounter{zdr}[subsection]

\newcommand{\eps}{\varepsilon}

\def\ve{\varepsilon}

\def\F{{\cal F}}
\def\pa{\partial}
\def\ve{\varepsilon}

\def\cal{\mathcal}
\let\mib=\boldsymbol
\def\ae#1{\;(\hbox{\rm a.e. } #1)}

\def\R{{\bf R}}

\def\N{{\bf N}}

\def\ae#1{\;(\hbox{\rm ae } #1)}

\def\Bnor#1#2{\Bigl\| #1 \Bigr\|_{#2}}
\def\Cbc#1{{{\rm C}^{\infty}_{c}(#1)}}

\def\CC{{\rm C}}

\def\Cc#1{{{\rm C}_{{\rm c}}(#1)}}
\def\Cnl#1{{{\rm C}_0(#1)}}

\def\Cp#1{{{\rm C}(#1)}}
\def\pC#1#2{{{\rm C}^{#1}(#2)}}

\def\dscon{\relbar\joinrel\rightharpoonup}

\def\Dup#1#2{\langle#1,#2\rangle}
\def\Dupp#1#2{\Bigl\langle#1,#2\Bigr\rangle}
\def\eps{\varepsilon}

\def\Lb#1{{{\rm L}^\infty(#1)}}
\def\Lbc#1{{{\rm L}_{{\rm c}}^\infty(#1)}}
\def\Ld#1{{{\rm L}^{2}(#1)}}
\def\Ldc#1{{{\rm L}^{2}_{{\rm c}}(#1)}}
 
\def\Ldws#1{{{\rm L}^{2}_{{\rm w^\ast}}(#1)}}
\def\Lj#1{{{\rm L}^{1}(#1)}}
\def\Ljl#1{{{\rm L}^{1}_{{\rm loc}}(#1)}}
\def\Ll#1#2{{{\rm L}^{#1}_{{\rm loc}}(#2)}}
\def\LL#1{{{\rm L}^{#1}}}

\def\LLd{{{\rm L}^2}}
\def\LLb{{{\rm L}^{\infty}}}
\def\pL#1#2{{{\rm L}^{#1}(#2)}}
\def\pLc#1#2{{{\rm L}^{#1}_{{\rm c}}(#2)}}
\def\pLl#1#2{{{\rm L}^{#1}_{{\rm loc}}(#2)}}

\def\lS{{\cal S}}

\def\malpha{{\mib \alpha}}
\def\mbeta{{\mib \beta}}

\def\mkappa{{\mib \kappa}}

\def\msnop{{\bf p}}

\def\mx{{\bf x}}
\def\mxi{{\mib \xi}}
\def\my{{\bf y}}

\def\Nnul{{\bf N}_0}
\def\Nor#1{\| #1 \|}    
\def\nor#1#2{{\| #1 \|}_{#2}}   
\def\oi#1#2{\langle#1,#2\rangle}
\def\ozi#1#2{\langle#1,#2]}
\def\Pd{{\rm P}}
\def\ph{\varphi}

\def\Rd{{{\bf R}^{d}}}
\def\Rm{{{\bf R}^{m}}}

\def\Rpln{{{\bf R}^{+}_0}}

\def\str{\longrightarrow}
\def\supp{{\rm supp\,}}
\def\Svaki#1{\left(\forall\,#1\right)}
\def\ve{{\sf e}}

\def\W#1#2#3{{{\rm W}^{#1,#2}(#3)}}

\def\zi#1#2{[#1,#2]}
\def\zoi#1#2{[#1,#2\rangle}
\begin{document}
\title[Optimal velocity averaging]{On the velocity averaging for equations with optimal heterogeneous rough coefficients}

\author{ Martin Lazar}
\address{Martin Lazar \newline
University of Dubrovnik, Department of Electrical Engineering and Computing,   Dub\-rov\-nik, Croatia }

 \email{  martin.lazar@unidu.hr} \date{}
 \author{ Darko Mitrovi\'c}
\address{ Darko Mitrovi\'c  \newline
University of Montenegro, Faculty of Mathematics, Podgorica, Montenegro}
 \email{  matematika@t-com.me}

 \date{}

\begin{abstract}
Assume that $(u_n)$ is a sequence of solutions to heterogeneous
equations with rough coefficients and fractional derivatives, weakly
converging to zero in ${\rm L}^p(\R^{d+m})$, with $p>1$.
 We prove that  the
sequence of averaged quantities $(\int \rho(\my) u_n(\mx,\my) d\my)$
is strongly precompact in $\Ljl\Rd$ for any $\rho\in \Cc{\R^m}$,
provided that restrictive non-degeneracy conditions are satisfied.
These are fulfilled for elliptic, parabolic, fractional
convection-diffusion equations, as well as for parabolic equations
with a fractional time derivative. The main tool that we are using
is an adapted version  of H-distributions. As a consequence of the
introduced methods, we obtain an optimal velocity averaging result
in the $\LL p$, $p\geq 2$, framework under the standard
non-degeneracy conditions, as well as a connection between the
H-measures and the H-distributions.
\end{abstract}



\maketitle

\section{Introduction}


In the paper we extend results from  \cite{LM2}
concerning the velocity averaging for a general transport-type
equations to an $\LL p$ setting with an arbitrary $p>1$. A simplified version of some results presented here, has been outlined, mostly without proof, in \cite{LM_crass}.

Accordingly, we consider a sequence of functions $(u_n)$ weakly
converging to zero in ${\rm L}^p(\R^d_\mx \times \R^{m}_\my)$ for
some $p> 1$, and satisfying the following sequence of equations

\begin{equation}
\label{general}
\begin{split}
{\cal P}u_n(\mx,\my)&=\sum\limits_{k\in
I}\partial^{\malpha_k}_\mx \left(a_k(\mx,\my)
u_n(\mx,\my)\right)=\pa^\mkappa_\my G_n(\mx,\my),
\end{split}
\end{equation} where $I$ is a finite set of indices,  $\partial_\mx^{\malpha_k}=\partial_{x_1}^{\alpha_{k1}}\dots \partial_{x_d}^{\alpha_{kd}}$
for a multi-index $\malpha_k=(\alpha_{k1},\dots,\alpha_{kd})\in
\left(\Rpln\right)^d$, and similarly for $\mkappa=(\kappa_1,\dots,\kappa_m)\in \Nnul^m$. The fractional derivative $\pa_{x_k}^{\alpha_k}u$ is
defined by
$$
\pa_{x_k}^{\alpha_k}u=\bar{\F}((2 \pi i \xi_k)^{\alpha_k} \hat{u}),
$$ where $\hat{u}(\mxi)=\F(u)(\mxi)=\int_{\R^d}e^{-2\pi i \mx\cdot
\mxi}u(\mx)d\mx$ is the Fourier transform while $\bar{\F}$ (or
$^\vee$) is the inverse Fourier transform. Remark that in our
setting, the fractional derivative is actually the Fourier
multiplier operator with the symbol $(2 \pi i \xi_k)^{\alpha_k}$
(see Definition \ref{multiplier}).

Denote by $A$ the principal symbol of the (pseudo-)differen\-tial
operator ${\cal P}$, which is of the form
\begin{equation}
\label{glavni}
 A(\mx,\my, \mxi)=\sum\limits_{k\in I'} a_k(\mx,\my) (2\pi
i\mxi)^{\malpha_k}.
\end{equation}
The sum given above is taken
over all terms from \eqref{general} whose order of derivative
$\malpha_k$ is not dominated by any other multi-index from $I$.

For the principal symbol we  assume that there exists a multi-index $\mbeta=(\beta_1, \cdots,
\\ \beta_d)\in \R^d_+$ such that for any positive $\lambda\in \R$  the
following generalised homogeneity assumption holds
\begin{equation}
\label{gen_homog}
A(\mx,\my,\lambda^{1/\beta_1}\xi_1,\dots,\lambda^{1/\beta_d}\xi_d)=
\lambda A(\mx,\my, \mxi),
\end{equation}
 implying that
 \begin{equation}
 \label{sum1}
 \sum_j {(\alpha_{kj} / \beta_j)}=1 \ \ \text{for every $k\in I'$}.
 \end{equation} In addition we assume that the order of derivatives $\alpha_{kj}$
entering the principal symbol  are either integers, or larger or
equal to the space dimension $d$.

As for the coefficients from \eqref{general}, we assume that

\begin{itemize}

\item [\bf a)]
For some $\bar{p}\in \oi 1p$, it holds
$$
a_k\in {\rm L}^{{p}'}(\R^{m};{\rm L}^{\bar{p}'}(\R^d)), \ \ k=1,\dots,d, \ \
\frac{1}{p}+\frac{1}{\bar{p}'}+\frac{1}{q}=\frac{1}{p}+\frac{1}{p'}
=1,
$$
where $q={p \bar p}/{(p- \bar p)}$, while $p'$ stands for a
dual index of $p$.

\item [\bf b)] The sequence $(G_n)$ is strongly precompact in the
anisotropic space \\
${\rm L}^{1}(\R^m;{\rm W}^{(-\beta_1,\dots,-\beta_d),q'}(\R^d))$.

\end{itemize}
We see that the coefficients $a_k$ are chosen in such a way that
the sequences $(a_k u_n)$ are bounded in ${\rm L}^{1}(\R^{m};{\rm L}^{1+\eps}(\R^d))$ for some $\eps>0$. We are not able to prove
that the velocity averaging result holds if $(a_k u_n)$ are merely
bounded in ${\rm L}^{1}(\R^{d+m})$, but the results from \cite{32}
hint that it is not possible to propose better assumptions than those given in ${\bf
a)}$ (unless  having additional requirements on $(u_n)$ as in
\cite{18}).

Let us now introduce a definition of the weak solution to
\eqref{general}. Assume for the moment that the subindex $n$ is
removed from \eqref{general}.

\begin{definition}
\label{weaksol} We say that a function $u\in \pL p{\R^{d+m}}$ is a
weak solution to \eqref{general} if for every $g\in
\Cbc{\R^m;\W\mbeta q\Rd}$  it holds
\begin{equation}
\begin{split}
\label{defws} &\int\limits_{\R^{m+d}}\!\sum\limits_{k\in I}
a_k(\mx,\my)
u(\mx,\my)\overline{(-\pa_{\mx})^{\malpha_k}g(\mx,\my)}d\mx d\my =
(-1)^{|\mkappa|}\!\!\int\limits_{\R^{m}}\!\Dupp{ G(\cdot,\my)}
{\overline{\pa^{\mkappa}_\my g(\cdot,\my)}} d\my\,,
\end{split}
\end{equation}
where in the last term duality on  $\W{\mbeta}{q}\Rd$ is considered.
\end{definition}

It has been noticed since long time ago that even in the case of
homogeneous coefficients \cite{Ago}  one cannot expect the very
sequence $(u_n)$, but only the associated sequence of its averages
with respect to the velocity variable to be strongly precompact in
${\rm L}^p_{{\rm loc}}(\R^{d})$. More precisely, it was proved in
\cite{LM2} for $p\geq 2$ (or in \cite{Ger} in the hyperbolic case)
that an averaged quantity
\begin{equation}
\label{averaged}
\big(\int_{\R^m}\rho(\my)u_n(\mx,\my)d\my\big), \
\ \rho\in \Cc{\R^m},
\end{equation}
where $(u_n)$ are solutions to \eqref{general},
will be strongly ${\rm L}^2_{{\rm loc}}$ precompact provided the following
non-degeneracy condition is fulfilled

\begin{equation}
\label{kingnl} \Svaki{(\mx, \mxi) \in D\times\Pd} \quad A(\mx,\my,\mxi)\not= 0
\quad\ae{  \my \in \R^{m}}\,,
\end{equation}
where $D\subseteq\Rd$ is a full measure set, while $\Pd$ stands for an appropriate $d-1$ dimensional compact manifold in $\Rd$.

A result of this type is usually called a velocity
averaging lemma.

Its importance is demonstrated in many works, but we shall mention
only very famous \cite{LPT} and \cite{8}. Concerning the averaging
lemma itself, there are also indeed interesting works \cite{13, 18,
32, 36}, but almost all of them were given for homogeneous equations
(i.e.~the ones where coefficients do not depend on $\mx\in \R^d$),
exclusively with integer-order derivatives. The reason for this one
can search in the fact that, in the homogeneous situation, one can
separate the solutions $u_n$ from the coefficients (e.g. by applying
the Fourier transform with respect to $\mx$), and this is basis of
most of the methods (see e.g. \cite{36} and references therein). We
remark that more detailed observations on this issue one can find in
the introduction of our recent work \cite{LM2}.

In order to attack the heterogeneous situation a different tool is
required, and it was provided independently by P.~Gerard \cite{Ger}
and L.~Tartar \cite{Tar} through the concept of microlocal defect
measures (in the terminology of the former), or H-measures (as named
by the latter).\footnote{In the sequel we shall use the terminology of H-measures.}  After their work, different variants of the concept
appeared, adapted to a problem under consideration \cite{AL, LM2,
MI, pan_arma}.

In \cite{Ger} one can find a velocity averaging lemma in a
heterogeneous setting proved by using H-measures. In \cite{LM2} we proved a general version of the velocity
averaging lemma. However, in both papers, the sequence $(u_n)$ was
bounded in ${\rm L}^p(\R^{d+m})$ for $p\geq 2$ since the
notion of H-measures  is defined in the $\LLd$ framework  (i.e.~they
describe a loss of compactness for such sequences). Such $\LL 2$
character of the H-measures provides their non-negativity, which
enables the authors  to conclude that given object is actually a
Radon measure defined on $\R^d \times \Pd$, for an appropriate $d-1$
dimensional manifold $\Pd$ (Gerard and Tartar worked with
$\Pd=S^{d-1}$, where $S^{d-1}$ is the unit sphere in $\R^d$). If the
$\Rd$ projection of the H-measure associated to the sequence $(u_n)$
equals zero, then the sequence $(u_n)$ is strongly precompact in
${\rm L}^2_{{\rm loc}}(\R^d)$. Remark that the
$\R^d$ projection is actually the standard defect measure \cite{Liocc}.

In order to overcome the mentioned $p\geq 2$  confinement we must
invent a more sophisticated tool. It will be based on a
generalisation of the H-distributions concept from \cite{AM}. The
H-distributions were introduced in order to describe a defect of
strong convergence for ${\rm L}^p(\R^d;\R^n)$ sequences (i.e.~for
$n$-dimensional ${\rm L}^p$ sequences). It is not difficult to
generalise such a concept when the sequences have countable
dimension, and even when functions assume values in a separable
Hilbert space $H$, i.e.~$u_n\in {\rm L}^p(\R^d;H)$ (Proposition
\ref{prop_reprAp}). For a further clarification, one can compare
works \cite{Ger} and \cite{Tar}, and also to consult
\cite[Proposition 12]{LM2}.

Thus in the case of $\LL p$ sequences $(u_n)$ for $p<2$, we have
merely a distribution (instead of a measure) describing eventual
loss of strong precompactness. Therefore, in order to use the
H-distributions on the velocity averaging problem for
\eqref{general}, we must increase assumptions on the principal
symbol, and we require the following restrictive non-degeneracy
condition
\begin{equation}
\label{rndc}
\ae {(\mx,\my)\in \R^{d+m}} \quad A(\mx,\my,\mxi)\not=0, \quad \forall {\mxi \in \Pd}.
\end{equation}
The  condition implies (actually it is equivalent to) the following strong convergence
\begin{equation}
\label{ggnc} \frac{|A|^{2}}{|A|^{2}+\delta} \str 1 \ \ \text{ in \ $
\Ll{p'}{\R^m;\Ll{\bar{p}'}{\R^d; \pC{d}\Pd}}$}\,,
\end{equation}
 as $\delta\to 0 $, which is  needed for the proof of the main theorem.

Indeed, if \eqref{rndc} holds, then for almost every $(\mx,\my)\in
\R^{d+m}$ we have that
$$
\Bnor{1-\frac{|A|^{2}}{|A|^{2}+\delta}}{\pC{d}\Pd} =\delta \Bnor{\frac{1}{|A|^{2}+\delta}}{\pC{d}\Pd},
$$
which goes to zero as $\min_\mxi |A(\mx,\my,\mxi)|>0$. Note that the fractions above are smooth enough due to the assumptions on order of derivatives entering the principal symbol $A$.

Physically relevant equations satisfying the restrictive
non-degeneracy conditions are elliptic and parabolic equations, but
also fractional convection-diffusion equations \cite{CJ, CT}, and
parabolic equations with a fractional time derivative \cite{BWM,
BWM1, dentz} which degenerate on a set of measure zero.

The paper is organised as follows.

In Section 2, we introduce auxiliary notions and notations.

In Section 3, we introduce a variant of the H-distributions required to prove the main result of the paper -- the velocity averaging lemma for
\eqref{general} under the assumption \eqref{ggnc}. Unlike the proof sketched in
\cite{LM_crass}, here we propose a different approach  which gives
rise to the connection between the H-measures and the
H-distributions (given in the Appendix), but also incorporates methods that were
previously applied to elliptic problems with singular data \cite{DHM}.

In \cite{LM2} we have proved that in the case $p\geq2$ the  velocity averaging result holds under the classical non-degeneracy
condition \eqref{kingnl} merely. Due to the H-measures techniques used there, the coefficients $a_k$ in
\eqref{general} are restricted to $\Ld{\R^{m};\pL r\Rd}$ where $2/p+1/r=1$. In
Section 4, we combine  methods developed in previous sections with the H-measures to
improve the velocity averaging result from \cite{LM2} in the sense
that the coefficients to \eqref{general} belong to ${\rm L}^{{p}'}(\R^{m};{\rm L}^{\bar{p}'}(\R^d))$
where $1/p+1/{\bar p}<1$.

\section{Notions and notations}
We start with the notion of the Fourier multiplier which forms the basis of
the current contribution.

\begin{definition}
\label{multiplier} A (Fourier) multiplier operator ${\cal A}_\psi:\Ld\Rd\to
\Ld\Rd$ associated to a bounded function $\psi$ (see e.g.
\cite{Gra}), is a mapping  defined by
$$
{\cal A}_\psi(u)=\bar{\F}(\psi \hat{u}),
$$where $\hat{u}$ is the Fourier transform while $\bar{\F}$ (or $^\vee$)
is the inverse Fourier transform.

If, for a given $p\in\zoi 1\infty$, the multiplier operator  ${\cal A}_\psi$
satisfies
$$
\nor{{\cal A}_\psi (u)}{\LL p} \leq C \nor{ u}{\LL p}, \qquad u\in
\lS,
$$
where $C$ is a positive constant, while $\lS$ stands for a Schwartz space, then its symbol $\psi$ is called
an $\LL p$ (Fourier) multiplier.
\end{definition}

We shall analyse multipliers  defined on the manifold
$\Pd$ determined by the order of the derivatives entering the principal symbol  \eqref{glavni}:
\begin{equation*}
\Pd=\{\mxi\in \R^d: \; \sum\limits_{i=1}^d |\xi_i|^{l \beta_i}=1 \},
\end{equation*} where $\mbeta$ is the  homogeneity index from
\eqref{gen_homog}, while $l$ is a minimal number such that either $l \beta_i >d$ or $l \beta_i$ is an even integer for each $i$. These assumptions ensure that the introduced manifold is of class $C^d$ which enables us to analyse associated multipliers, as well as to define appropriate variant of the
H-distributions on them (see Theorem \ref{prop_repr}).

In order to associate an $\LL p$  multiplier to a function defined
on $\Pd$ we extend it to $\R^d\backslash \{0\}$ by means of the
projection
\begin{equation}
\label{projct} \big(\pi_{\Pd}(\mxi)\big)_i=\xi_i\,
\Big(|\xi_1|^{l \beta_1}+\dots+|\xi_d|^{l \beta_d} \Big)^{- 1/l
\beta_i}
=\xi_i\,
|\mxi|_\mbeta^{- 1/\beta_i}, \ \ i=1,\dots,d, \ \ \mxi\in \R^d\backslash\{0\}\,,
\end{equation}
where here and in the sequel we use abbreviation $|\mxi|_\mbeta=\Big(\sum |\xi_i|^{l \beta_i}\Big)^{1/l}$.

There are many criteria on a symbol $\psi$ providing it to be an
$\LL p$ multiplier.  In the paper, we shall need the Marcinkiewicz
multiplier theorem \cite[Theorem 5.2.4.]{Gra}, more precisely its
corollary which we provide here:

\begin{corollary}
\label{m1} Suppose that  $\psi\in \pC{d}{\R^d\backslash
\cup_{j=1}^d\{\xi_j= 0\}}$ is a bounded function such that for some constant $C>0$
it holds
\begin{equation}
\label{c-mar} |\mxi^{\tilde\malpha} \partial^{\tilde\malpha} \psi(\mxi)|\leq C,\ \
 \mxi\in \R^d\backslash \cup_{j=1}^d\{\xi_j= 0\}
\end{equation}
for  every multi-index
$\tilde\malpha=(\tilde\alpha_1,\dots,\tilde\alpha_d) \in \Nnul^d$ such that
$|\tilde\malpha|=\tilde\alpha_1+\tilde\alpha_2+\dots+\tilde\alpha_d \leq d$. Then, the function
$\psi$ is an ${\rm L}^p$-multiplier for $p\in \oi 1\infty$, and the
operator norm of ${\cal A}_\psi$ depends only on $C, p$ and $d$.
\end{corollary}

\begin{remark}
\label{bound}
Using this corollary, we have proved that for a bounded function $\psi$ defined on the manifold $\Pd$ and smooth outside coordinate hyperplanes, its extension $\psi_\Pd = \psi\circ\pi_{\Pd}$ is an $\LL p$ multiplier (see \cite[Lemma 5]{LM2}).
If in addition we assume that $\psi$ is smooth on the whole manifold, i.e. $\psi \in
\pC d\Pd$, then the corresponding operator satisfies
\begin{align}
\label{marz_bnd} \|{\cal A}_{\psi_\Pd} \|_{\LL p\to \LL p} \leq C
\|\psi\|_{\pC d\Pd},
\end{align}
with a constant $C$ depending only on $p\in \oi 1\infty$ and $d$.
\end{remark}
Here, we shall need a similar statement.

\begin{lemma}
\label{l1IIv}
Let  $\mbeta\in \R^d_+$  and let $\theta:\R^d\to \R$ be a smooth compactly
supported function equal to one on the unit ball centered at origin.

Then  for any $\gamma>0$ the multiplier operator ${\cal T}^{\gamma}$ with
the symbol
$$
T^\gamma(\mxi)(1-\theta(\mxi))=\frac{1}{|\mxi|_\mbeta^{\gamma}}(1-\theta(\mxi))
$$
is a continuous $\pL p\Rd\to {\rm W}^{\gamma \mbeta, p}(\R^d)$ operator for any $p\in \oi 1\infty$.
Specially, due to the Rellich theorem it is a compact $\pL p\Rd\to \pLl p\Rd$ operator.
\end{lemma}
\begin{proof}
We shall first prove that the operator ${\cal T}^\gamma$ is a continuous operator on $\pL p\Rd$. To this effect, remark that
it is enough to prove that $T^\gamma$ satisfies condition
of Theorem \ref{m1} away from the origin. Around the origin, the
operator ${\cal T}^\gamma$ is controlled by the term $(1-\theta)$ (which is
equal to zero on $B(0,1)$ and obviously satisfies conditions of
Theorem \ref{m1}).
We use the induction argument with respect to the order of derivative in \eqref{c-mar}.

\begin{itemize}

\item $n=1$

In this case, we compute
\begin{equation*}
\pa_{k}T^\gamma(\mxi)=
C_{k}\frac{1}{\xi_k} T^\gamma(\mxi) \big(\pi_{\Pd}(\mxi)\big)_k^{l \beta_k}
\end{equation*}
for some constant $C_{k}$. From here, it obviously
follows $|\xi_k \pa_{k} T^\gamma(\mxi)| \leq C$ for
$\mxi\in \R^d$ away from the origin.

\item $n=m$

Our inductive hypothesis is that a $\malpha$-order derivatives of $T^\gamma(\mxi)$
can be represented in the following way
\begin{equation}
\label{ih}
\partial^\malpha T^\gamma(\mxi)
=\frac{1}{\mxi^\malpha} T^\gamma(\mxi) P_\malpha(\mxi),
\end{equation}
where $P_\malpha$ is a bounded function satisfying \eqref{c-mar} for $|\tilde \malpha| \leq d - |\malpha|$.

\item $n=m+1$

To prove that \eqref{ih} holds for $|\malpha|=m+1$ it is enough to
notice that $\malpha=\ve_k+\malpha'$, where $|\malpha'|=m$, and that according to the induction hypothesis we have
\begin{align*}
&\partial^\malpha T^\gamma =\partial_k \partial^{\malpha'}
T^\gamma
=\partial_k\left(\frac{1}{\mxi^{\malpha'}} T^\gamma(\mxi) P_{\malpha'}(\mxi)\right)
=\frac{1}{\mxi^\malpha} T^\gamma(\mxi) P_\malpha(\mxi),
\end{align*}
where
$$
P_\malpha(\mxi)= (P_{\ve_k} P_{\malpha'}+ \xi_k \partial_k P_{\malpha'}
-\alpha_k P_{\malpha'}) (\mxi),
$$
thus satisfying conditions \eqref{c-mar} as well.

\end{itemize}

>From here, \eqref{c-mar} immediately follows for
$T^\gamma$ away from the origin, thus proving that the operator ${\cal T}^\gamma$ is a continuous operator on $\pL p\Rd$.

It remains to prove that for any $\beta_j$ from the $d$-tuple
$\mbeta=(\beta_1,\ldots,\beta_d)$, the multiplier operator
$\pa^{\gamma \beta_j}_{x_j} {\cal T}^\gamma$ is a continuous ${\rm L}^p(\R^d)\to
{\rm L}^p(\R^d)$ operator. To accomplish this, notice that its symbol is
\begin{equation}
\label{1_2014}
(1-\theta(\mxi)) \frac{(2\pi i \xi_j)^{\gamma
\beta_j}}{|\mxi|^\gamma_{\mbeta}}
= (1-\theta(\mxi)) \big(\pi_{\Pd}(\mxi)\big)_j^{\gamma \beta_j}.
\end{equation}
Thus, away from the origin, it is a composition of a  function which is smooth outside coordinate hyperplanes and the projection $\pi_{\Pd}$, and by Remark \ref{bound}  satisfies conditions of Theorem \ref{m1}.
\end{proof}

In the paper we shall need the following generalisation of Tartar's commutation lemma \cite[Lemma 28.2]{tar_book} to $\LL p, p\not=2$ sequences.

\begin{lemma}
\label{CommLemma} Let $B$ be the operator of multiplication by a
continuous function $b\in \Cnl\Rd$. Let $(v_n)$ be a bounded
sequence in $\Ld\Rd \cap \pL p\Rd$, $p\in \zi 1\infty$ such that
$v_n \rightharpoonup 0$ in the sense of distributions, and let
$\psi\in \pC d\Pd$. Then for the commutator  $C={\cal A}_{\psi_\Pd}
B-B {\cal A}_{\psi_\Pd}$ the sequence $(C v_n)$ converges strongly
to zero in $\pL q\Rd$ for any $q\in [2,p]\backslash \{\infty\}$ if $p\geq 2$, and any $q\in [p,2]\backslash \{1\}$ if $p< 2$,.

\end{lemma}
\begin{proof}
In \cite[Theorem 6]{LM2} we have proved that for an arbitrary
$\psi\in \pC d\Pd$ the extension $\psi_\Pd$ satisfies the conditions
of Tartar's commutation lemma, thus ensuring    that $C$ is a
compact operator on $\Ld\Rd$.

According to the
interpolation inequality for any $r$ between 2 and $p$, and $\alpha\in\oi01$ we have
\begin{equation}
\label{pro_1} \|Cv_n\|_{q}\leq
\|Cv_n\|^\alpha_{2}\|Cv_n\|^{1-\alpha}_{r},
\end{equation}
where $1/q=\alpha/2+(1-\alpha)/r$. As $C$ is
a compact operator on $\Ld\Rd$, while $C$ is bounded on
$\pL r\Rd$ for $r\in \oi 1\infty$,  we get the claim.
\end{proof}

Next, we shall need the following truncation operator
\begin{equation}
T_l(u)=\begin{cases} 0, & |u|>l\\
u, & u\in [-l,l]
\end{cases}, \ \ l\in \N.
\label{trunc}
\end{equation} The operator (more precisely its variant) was introduced in
\cite{DHM} where it was noticed that convergence of $(T_l (u_n))$
for every $l\in \N$ in $\Ljl\Rd$ implies the strong
convergence of $(u_n)$ in $\Lj\Rd$ (Lemma \ref{DHM} below). This property will be used in order to
prove the strong precompactness of the averaged family
\eqref{averaged}. Moreover, the truncation operator will enable us
to obtain a relation between the H-measures and the
H-distributions (see the Appendix).

The following statements ensure the above mentioned property of the truncation operator.

\begin{lemma}
\label{trunc-l}
 Let $(u_n)$ be a bounded sequence in $\pL p\Omega$ for
some $p> 1$,  where $\Omega$ is an open set in $\R^d$.
Then for   the sequence of truncated functions it holds
\begin{equation}
\label{uni_n}
\lim_l \sup_n \nor{T_l(u_{n}) -  u_{n}}{\pL 1\Omega} \str 0\,.
\end{equation}
\end{lemma}
\begin{proof}
Denote by
$$
\Omega_n^l=\{ \mx\in\Omega:\, u_{n}(\mx) > l \}.
$$
Since $(u_{n})$ is bounded in $\pL p\Omega$ we have
\begin{align*}
&\sup\limits_{k\in \N} \int_{\Omega }|u_{n}(\mx)|^p d\mx \geq
\sup\limits_{k\in \N} \int_{\Omega_n^l} l^p d\mx\,,
\end{align*}
implying that
\begin{equation}
\label{conv1}
\lim\limits_{l\to \infty} \sup\limits_{k\in \N} {\rm
meas}(\Omega_n^l) =0.
\end{equation}
Now, we use the H\"{o}lder inequality
$$
\int_\Omega |u_{n}-T_l(u_{n})|dx=\int_{\Omega_n^l} |u_{n}|dx \leq
{\rm meas}( \Omega_n^l)^{1/p'} \|u_{n}\|_{\pL p{\Omega}}
$$
which tends to zero uniformly with respect to $n$ according to \eqref{conv1}.
Thus, \eqref{uni_n} is proved.
\end{proof}

\begin{lemma}
\label{DHM} Let $(u_n)$ be a bounded sequence in $\pL p\Omega$ for
some $p> 1$,  where $\Omega$ is an open set in $\R^d$. Suppose
that for each $l\in\N$ the sequence of truncated functions
$(T_l(u_n))$ is precompact in ${\rm L}^1(\Omega)$. Then there
exists a subsequence $(u_{n_k})$ and function $u\in \pL p\Omega$ such that
$$
u_{n_k} \str u  \ \ \text{in $\pL 1\Omega$.}
$$
\end{lemma}

\begin{proof}

Due to the strong precompactness assumptions on truncated sequences, there exists a subsequence  $(u_{n_k})$  such that for every $l\in \N$ the sequence $(T_l(u_{n_k}))$ is convergent in $\Lj \Omega$, with  a limit denoted by $u^l$.
We  prove that the obtained sequence $(u^l)$ strongly converges in
$\Lj \Omega$ as well.

To this end, note that
\begin{align}
\label{ul} \|u^{l_1}-u^{l_2}\|_{\Lj \Omega}
&\leq
\|u^{l_1}-T_{l_1}(u_{n_k})\|_{\Lj \Omega}+\|T_{l_1}(u_{n_k})-u_{n_k}\|_{\Lj \Omega}\nonumber\\
&+\|T_{l_2}(u_{n_k})-u_{n_k}\|_{\Lj \Omega}+\|T_{l_2}(u_{n_k})-u^{l_2}\|_{\Lj \Omega}\,,\nonumber
\end{align}
which together with  Lemma \ref{trunc-l} implies that $(u^l)$ is a Cauchy sequence. Thus, there exists $u\in \Lj \Omega$ such that
\begin{equation}
\label{conv2}
u^l \to u \ \ {\rm in} \ \ \Lj \Omega.
\end{equation}

Now it is not difficult to see that entire $(u_{n_k})$ converges toward $u$
in $\Lj \Omega$ as well. Namely, it holds

\begin{align*}
\label{dbk11} \|u_{n_k}-u\|_{\Lj \Omega} \leq
\|u_{n_k}-T_l(u_{n_k})\|_{\Lj \Omega}+\|T_l(u_{n_k})-u^l\|_{\Lj \Omega}+\|u^l-u\|_{\Lj \Omega},
\end{align*}
which by the definition of functions $u^l$, and convergences \eqref{uni_n} and \eqref{conv2} imply the statement.
\end{proof}

\section{H-distributions and velocity averaging}

We  start the section with description of the variant of
H-distributions that we use in the proof of the main theorem. It has
been recently introduced in \cite{LM_crass} in an isotropic case, and it is an extension
of the concept proposed in \cite{AM}.

\begin{theorem}
\label{prop_repr} Let $(u_n)$ be a bounded sequence in $\pL
p{\R^{d+m}}$, $p>1$, and let $(v_n)$ be a sequence converging weakly
to zero in $\Ld\Rd \cap {\rm L}^{q}(\R^d)$ for some (finite) $q \geq
p'$. Let $\bar{p}\in\zoi 1p$ be such that
$\frac{1}{p}+\frac{1}{\bar{p}'}+\frac{1}{q}=1$. Then, after passing
to a subsequence (not relabeled), there exists a continuous bilinear
functional $B$ on $\pL{p'}{\R^{m};\pL{\bar{p}'}\Rd} \otimes
\pC{d}\Pd$ (with $\Lb\Rd$ being replaced by $\Cnl\Rd$ if $q=p'$)
such that for every $\phi_1\in \pL {{p}'}{\R^{m}; \pL{\bar{p}'}\Rd}
$, $\phi_2 \in \Cnl\Rd$, and $\psi \in \pC{d}\Pd$, it holds
\begin{equation}
\label{rev1}
 B(\phi_1 \overline{\phi_2}, \overline\psi)=\lim\limits_{n\to
\infty}\int_{\R^{d+m}} \phi_1(\mx, \my)u_n(\mx,\my)
\overline{{\cal A}_{\psi_\Pd} \big(\phi_2 v_n\big)(\mx)} d\mx
d\my\,,
\end{equation}
where ${\cal A}_{\psi_\Pd}$ is the (Fourier) multiplier operator on
$\R^d$ associated to $\psi\circ \pi_\Pd$.

The functional $B$ we call the H-distribution corresponding to (sub)sequences (of) $(u_n)$ and $(v_n)$.
\end{theorem}

\begin{remark}
In the case $p=q= 2$,  the H-distribution defined above is
actually the (generalised) H-measure corresponding to sequences $(u_n)$ and
$(v_n)$ (see Theorem \ref{th10}).
\end{remark}

\begin{proof}
First, remark that according to the commutation lemma (Lemma
\ref{CommLemma}), it holds

\begin{align}
\label{rev2}
\lim\limits_{n\to \infty}&\int_{\R^{d+m}} \phi_1(\mx,
\my)u_n(\mx,\my) \overline{{\cal A}_{\psi_\Pd} \big(\phi_2
v_n\big)(\mx)} d\mx d\my\\
&=\lim\limits_{n\to \infty}\int_{\R^{d+m}} \phi_1(\mx,
\my)\overline{\phi_2(\mx)} u_n(\mx,\my) \overline{{\cal A}_{\psi_\Pd} (v_n)(\mx)}
d\mx d\my. \nonumber
\end{align}
Thus the limit in \eqref{rev1} depends only on the product $\phi_1 \overline{\phi_2}\in \pL {{p}'}{\R^{m}; \pL{\bar{p}'}\Rd}$.

Next,  consider the bilinear mapping $B_n$ defined for every $\phi
\in \pL{p'}{\R^{m};\pL{\bar{p}'}\Rd}$ and $\psi\in \pC{d}\Pd$ by
$$
B_n(\phi, \psi) = \int_{\R^{d+m}} \phi(\mx, \my) u_n(\mx,\my)
\overline{{\cal A}_{\psi}( v_n)(\mx)} d\mx d\my.
$$
According to the H\"{o}lder inequality and the Marcinkiewicz theorem,
it holds
\begin{equation*}
 |B_n(\phi, \psi)| \leq C
\|\psi\|_{\pC{d}\Pd} \| v_n \|_{\pL{q}{\R^d}}
\int_{\R^{m}}
\|\phi(\cdot,\my)\|_{\pL{\bar p'}{\R^d}} \| u_n(\cdot,\my)
\|_{\pL{ p}{\R^d}}  d\my,
\end{equation*} where $C$ is the constant from relation \eqref{marz_bnd} depending
on $d$ and $q$. By using the
H\"{o}lder inequality again (now applied in the variable $\my$), we get that

\begin{align}
\label{step2}
|B_n(\phi, \psi)| & \leq
 C \|\psi\|_{\pC{d}\Pd} \| v_n \|_{\pL{q}{\R^d}}
\|\phi\|_{\pL{p'}{\R^{m};\pL{\bar{p}'}{\R^d}}}
\|u_n\|_{\pL{ p}{\R^{d+m}}}  \nonumber\\
&\leq \bar{C} \|\psi\|_{\pC{d}\Pd}
\|\phi\|_{\pL{p'}{\R^{m};\pL{\bar{p}'}{\R^d}}} , \nonumber
\end{align}
where $\bar{C}$ depends on $C$, and bounds on $\|u_n\|_{\pL{ p}{\R^{d+m}}}$ and
$\|v_n\|_{\pL{q}{\R^d}}$.

Thus it follows that $(B_n)$ is an equibounded sequence of bilinear
functionals, and by \cite[Lemma 3.2]{AM} and \eqref{rev2}, there
exists a functional $B$ for which \eqref{rev1} holds.
\end{proof}

It is not difficult to see that the H-distributions given in the
previous theorem exhibit similar properties  as the H-measures, in
the sense that trivial H-distribution implies a velocity averaging
result. Indeed, for a fixed $l\in \N$ and $\chi\in \Lbc{\R^{d+m}}$
take $(v_n^l)=\big(T_l(\int_{\R^{d+m}}\overline{\chi(\mx,\tilde\my)
u_n(\mx,\tilde\my)}d\tilde\my)\big)$, where $(u_n)$ is a sequence from the last theorem. Suppose that the projection on $\Rd$ of
the H-distribution $(B_l)$ corresponding to the sequences $(u_n)$
and $(v_n^l)$ is equal to zero for every $l\in \N$, i.e.
$$
B_l(\phi, 1)=0, \ \
\phi \in \pL{p'}{\R^{m};\pL{\bar{p}'}{\R^d}}, \;  l \in \N.
$$
By choosing $\phi_1 \overline{\phi_2}=\chi$ in \eqref{rev1}, we
conclude that
\begin{equation*}
\lim\limits_{n\to \infty}\int_{\R^{d}} \left|T_l \int_{\R^m}\chi(\mx,
\my)u_n(\mx,\my)d\my\right|^2 d\mx =0.
\end{equation*}
>From here and Lemma \ref{DHM}, we get that $(\int_{\R^m}\chi(\mx,
\my)u_n(\mx,\my)d\my)$ converges to zero strongly in $\Ljl\Rd$. From
the interpolation inequalities, the sequence is also strongly
precompact in $\Ll{\bar{p}}{\R^d}$ for any $\bar{p}<p$.

More information concerning the connection between the H-measures
and the H-distributions one can find in the Appendix.

Now, we go back to the main subject of the paper -- the velocity
averaging result and tools required for its proof. Remark that
according to the Schwartz kernel theorem \cite{Sw}, one can extend
functional $B$ to a distribution on ${\cal D}'(\R^{d+m}\times \Pd)$.
An improved result can be obtained by means of the next theorem. It is a
generalisation of \cite[Theorem 2.1]{LM_crass} to Lebesgue spaces
with mixed norms, with the proof going along the same lines. The
considered anisotropic Lebesgue spaces $\pL {\msnop}\Rd$ are Banach
spaces with the norm given by
$$
\nor f\msnop = \Bigl(\int \cdots\Bigl(\int\Bigl(\int|f(x_1,x_2,
\dots, x_d)|^{p_1} d x_1\Bigr)^{p_2/p_1} d x_ 2\Bigr)^{p_3/p_2}
\dots d x_ d\Bigr)^{1/p_d}\,.
$$
\begin{theorem}
\label{thm1} Let $B$ be a  continuous bilinear functional  on $\pL {\msnop}\Rd\otimes E$,
where $E$ is  a separable Banach space, and $\msnop\in {\oi 1\infty}^d$. Then
$B$  can be extended as a continuous functional on $\pL
{\msnop}{\R^{d};E}$ if and only if there exists a (nonnegative) function
$b\in \pL{{\msnop'}}\Rd$ such that for every $\psi\in E$ and almost every
$\mx\in\Rd$, it holds
\begin{equation}
\label{cond} | \tilde B \psi(\mx)| \leq b(\mx) \|\psi\|_{E} \,,
\end{equation}
where $\tilde B $ is a bounded linear operator $E \to \pL{{\msnop'}}\Rd$
defined by $\Dup{\tilde B \psi}\phi =B(\phi, \psi)$, $\phi \in
\pL{{\msnop}}\Rd$.

\end{theorem}
The proof in \cite{LM_crass}  is presented for real functionals, but the result holds for complex ones as well, with the proof   going along the same lines  (just by considering the real and imaginary part of $\tilde B$ separately).

We use the theorem in order to get the following result for the functional $B$ from Theorem \ref{prop_repr}.

\begin{corollary}
\label{korolar}

If $\bar p>1$  the bilinear functional $B$ defined in Theorem
\ref{prop_repr}  can be extended as a continuous functional on $\pL{p'}{\R^{m};\pL{\bar{p}'}{\Rd; \pC d\Pd}}$.
\end{corollary}

\begin{proof} We shall prove that the functional $B$ satisfies conditions of
Theorem \ref{thm1}.

 As the first step, choose a dense countable set $E$
of functions $\psi_j$ on the unit sphere in $\pC d\Pd$. In addition,
we assume that for each such function, $-\psi_j$ belongs to the same
set as well and is indexed by $-j, j\in \N$.

Let $\tilde B$  be an operator defined in Theorem \ref{thm1}, which in this setting is a bounded linear operator
$\pC d\Pd \to
\pL{p}{\R^{m};\pL{\bar{p}}{\Rd}}$.
For each  function $\tilde B\psi_j$ denote by
$D_j$ the corresponding set of Lebesgue points. The set $D_j$ is of
full measure, and thus the set $D=\cap_j D_j$  as
well.

For any $(\mx, \my)\in D$ and $k\in\N$ denote
\begin{align}
\label{bk1}
b_k^{\cal R}(\mx, \my):=&\max\limits_{|j| \leq k} \Re(\tilde
B\psi_j (\mx, \my) )
=\sum\limits_{|j|=1}^k  \Re(\tilde B \psi_j
(\mx, \my)) \chi_j^k (\mx, \my)\\
b_k^{\cal I}(\mx, \my):=&\max\limits_{|j| \leq k} \Im(\tilde
B\psi_j (\mx, \my) )
=\sum\limits_{|j|=1}^k  \Im(\tilde B \psi_j
(\mx, \my)) \tilde\chi_j^k (\mx, \my)\,,\nonumber
\end{align}
where $\chi^k_{j}, \tilde\chi^k_{j}$, $|j|=1,\dots,k$ are characteristic functions of the sets of all points  for which the above maxima are achieved
for $\psi_{j}$.

We shall prove that the  functions $b_k^{\cal R}, b_k^{\cal I}$ given by \eqref{bk1}
are uniformly bounded in $\pL{p}{\R^{m};\pL{\bar{p}}{\Rd}}$. Since
 $(b_k^{\cal R}), (b_k^{\cal I})$ are increasing sequence of positive functions,  denoting by $b^{{\cal R}},b^{{\cal I}}$ their (pointwise) limits, the function  $b=b^{{\cal R}}+b^{{\cal I}}$ will satisfy the conditions of Theorem \ref{thm1}.
Indeed, according to \eqref{bk1}, we see that \eqref{cond} will hold for every $\psi_j \in E$, and by continuity, the statement
can be generalised to an arbitrary $\psi\in \pC d\Pd$.

Thus it remains to prove the boundedness of the sequence $(b_k^{\cal R}), (b_k^{\cal I})$. We write down the proof just for the first one, as for the second one is performed in completely the same way.   To this effect, take an arbitrary $\phi\in  \Cc{\R^{d+m}}$, and
denote $K=\supp \phi$. Let $\chi^{k, \eps}_{j}\in
\Cc{\R^{d+m}}$ be smooth approximations of
characteristic functions from \eqref{bk1} on $K$ such that

$$
\|\chi^{k, \eps}_{j}-\chi^k_{j} \|_{\pL {r}{K}}\leq \frac{\eps}{k},
$$
where $r>1$ is chosen such that $q=r' p'$.
Denote by $C_u$ an $\LL p$ bound of $(u_n)$ and by $C_{v}$ an $\LL {q}$ bound of $(v_n)$ .

According to \eqref{bk1} and the definition of the operator $\tilde B$, we have
\begin{align*}
&\big|\langle b_k^{\cal R}, \phi \rangle\big| =\Big| \lim\limits_{n\to
\infty}\Re\Big(\int_{\R^{d+m}} \sum\limits_{|j|=1}^k (\phi u_n  \chi^k_j)
(\mx, \my) ( {\cal A}_{\psi_j} v_n)(\mx) d\mx d\my\Big) \Big|
\\&\leq \limsup\limits_{n\to \infty} \int_{\R^{d+m}}
\Big(\sum\limits_{|j|=1}^k |u_n|^p \chi^k_j  \, (\mx, \my)\Big)^{1/p} \Big(
\sum\limits_{|j|=1}^k \chi_j^k |\phi\, {\cal A}_{\psi_j}
v_n|^{p'}\, (\mx, \my)\Big)^{1/p'} d\mx d\my
\\&\leq \limsup\limits_{n\to \infty}\|\sum\limits_{|j|=1}^k |u_n|^p
\chi^k_j\|^{1/p}_{\pL 1{\R^{d+m}}} \|\sum\limits_{|j|=1}^k \chi^k_j
|\phi\, {\cal A}_{\psi_j}v_n|^{p'}\|^{1/p'}_{\pL 1{\R^{d+m}}}
\\&\leq\limsup\limits_{n\to \infty} \|u_n\|_{\pL p{\R^{d+m}}} \, \Big(\|
\sum\limits_{|j|=1}^k (\chi^k_j-\chi^{k, \eps}_j) |\phi\, {\cal A}_{\psi_j}v_n|^{p'}\|_{\pL {1}{\R^{d+m}}}
\nonumber\\&\qquad\qquad\qquad\qquad\qquad\qquad+
\|\sum\limits_{|j|=1}^k \chi^{k, \eps}_j \phi\, {\cal A}_{\psi_j}v_n\|^{p'}_{\pL {p'}{\R^{d+m}}} \Big)^{1/p'}\\
&\leq  C_u \limsup\limits_{n\to \infty}\Big(\sum\limits_{|j|=1}^k \|
\chi^k_j-\chi^{k, \eps}_j \|_{\pL {r}K}  \| {\cal A}_{\psi_j}(\phi
v_n
)\|^{p'}_{\pL {q}{\R^{d+m}}}\\&\qquad\qquad\qquad\qquad\qquad\qquad+\sum\limits_{|j|=1}^k\|{\cal A}_{\psi_j}(\chi^{k, \eps}_j \phi
v_n)\|^{p'}_{\pL {p'}{\R^{d+m}}} \Big)^{1/p'},
\end{align*}
where in the last step we have used the  commutation lemma (Lemma
\ref{CommLemma}) and that $r' p'=q$. By means of the Marcinkiewicz theorem and
properties of the functions $\chi^{k, \eps}_j$ it follows
\begin{align*}
&\big|\langle b_k^{\cal R}, \phi \rangle\big| \leq C_u C_{p'}\limsup
\limits_{n\to \infty}\Big(\eps \, C_\phi \|v_n
\|^{p'}_{\pL {q}\Rd}\!+\!\sum\limits_{|j|=1}^k \|\chi^{k, \eps}_j \phi
v_n\|^{p'}_{\pL {p'}{\R^{d+m}}}\!\Big)^{\!1/p'}\hskip -3mm,
\end{align*}
where $C_{p'}$ is the constant from the corollary of the
Marcinkiewicz theorem (more precisely from \eqref{marz_bnd}), while $C_\phi=\|\phi
\|^{p'}_{\pL {q}{\R^{m}; \Lb\Rd}}$. By
letting  $\eps \to 0$, we conclude
\begin{align*}
\big|\langle b_k^{\cal R}, \phi \rangle\big|
\leq C_u C_{p'}\limsup \limits_{n\to \infty}
\Big(\sum\limits_{|j|=1}^k \|\chi^{k}_j \phi v_n\|^{p'}_{\pL {p'}{\R^{d+m}}}\Big)^{1/p'}
\end{align*}
since $\chi^{k, \eps}_j\to \chi^k_j$ in $\pL {\tilde r}K$ for any $\tilde r\in\zoi 1\infty$. Thus, as
$$
\sum\limits_{|j|=1}^k \|\chi^{k}_j \phi v_n\|^{p'}_{\pL {p'}{\R^{d+m}}}=\|\phi
v_n\|^{p'}_{\pL {p'}{\R^{d+m}}}\leq
\left(\|\phi\|_\pL{p'}{\R^{m};\pL{\bar{p}'}\Rd} \|v_n\|_{\pL {q}\Rd}\right)^{ p'}\, ,
$$
it follows
\begin{equation*}
\lim\limits_{k\to \infty}\big|\langle b_k^{\cal R}, \phi \rangle\big| \leq C_u C_{p'}C_{v} \|\phi\|_{\pL{p'}{\R^{m};\pL{\bar{p}'}\Rd}}.
\end{equation*}
Since $\Cc{\R^{d+m}}$ is dense in $\pL{p'}{\R^{m};\pL{\bar{p}'}\Rd}$
we conclude that the sequence $(b_k^{\cal R})$ is bounded in
$\pL{p}{\R^{m};\pL{\bar{p}}\Rd}$. Concluding the same for  $(b_k^{\cal I})$, and denoting respectively the limits by $b^{\cal R}, b^{\cal I}$, one gets that  $b=b^{\cal R}+ b^{\cal I}$
satisfies \eqref{cond}. The result now follows from Theorem
\ref{thm1}.
\end{proof}

Now, we can prove the main result of the paper.

\begin{theorem}
\label{t-main} For the sequence of equations \eqref{general} we assume
\begin{itemize}

\item the coefficients of
satisfy conditions ${\bf a)}$, ${\bf b)}$;

\item the principal symbol \eqref{glavni} satisfies the homogeneity assumption \eqref{gen_homog} and the restrictive
non-degeneracy condition \eqref{rndc};

\item $u_n \dscon 0$ in $\pL{p}{\R^{d+m}}$, for some
$p>1$.
\end{itemize}
Then for any $\rho\in \Cc{\R^m}$ the sequence of averaged
quantities $\big(\int_{\R^m} \rho(\my) u_n(\mx,\my)d\my\big)$
 converges to 0 strongly in $\Ljl\Rd$.
\end{theorem}


\begin{proof}

Fix $\rho\in \Cc{\R^m}$, $\ph\in\Lbc\Rd$,  and $l\in \N$. Denote by $V_l$ a weak
$\ast$ ${\rm L}^\infty(\R^{d+m})$ limit along a subsequence of truncated averages defined by $V^l_n=\ph T_l ( \int_{\R^m} \rho(\tilde\my)
u_n(\cdot,\tilde\my) d\tilde\my)$, where $T_l$ is the truncation operator introduced in \eqref{trunc}. Denote $v^l_n=V^l_n-V_l$ and remark
that $v^l_n\buildrel \ast \over \dscon 0$ in
${\rm L}^\infty(\R^{m})$ with respect to $n$.

Next, let $B_l$ be the H-distribution defined in Theorem
\ref{prop_repr} corresponding to  (sub)se\-quen\-ces (of) $u_n$ and $v^l_n$.


Take a dual product of \eqref{general} with the test function $g_n$
of the form (${\cal T}^1$ below is defined in Lemma \ref{l1IIv})
\begin{equation*}
g_n(\mx,\my)=\rho_1(\my) ({\cal T}^1\circ {\cal
A}_{\psi_{\Pd}}) (\ph_1 v_n)(\mx),
\end{equation*}
where  $\psi\in \pC d\Pd$, $\ph_1\in \Cbc\Rd$, and $\rho_1 \in {\rm
C}_c^{|\mkappa|}(\R^m)$ are arbitrary test functions, while $\mkappa$ is the multi-index appearing
in \eqref{general} (see \eqref{defws}). We get

\begin{align*}
\sum\limits_{k\in I}\int_{\R^{d+m}}
a_k(\mx,\my)u_n(\mx,\my)
\bar\rho_1(\my)
\overline{{\cal A}_{\psi_{\Pd}}\circ{\cal
A}_{(1-\theta(\mxi))\frac{
(-2\pi i \mxi)^{\malpha_{k}}}{|\mxi|_{\mbeta}}}(\ph_1 v_n)(\mx)}
d\mx d=
o_n(1),
\end{align*}
where $|\mxi|_{\mbeta}$ is defined in   \eqref{projct}. The right-hand side of the last expression tends to zero as $n\to \infty$ as, by assumption {\bf b)}, the sequence
 $(G_n)$ of functions on the right hand side of \eqref{general}
converges strongly to zero in $\Lj{\R^m;\W{-\mbeta}{q}{\R^d}}$, while, according to Lemma \ref{l1IIv}, the
multiplier operator ${\cal T}^1 \circ {\cal A}_{\psi_{\Pd}}:
\pL{q}\Rd \to \W\mbeta{q}{\R^d}$ is bounded .

Rewriting the last relation  and passing to the limit we get
\begin{align}
\label{bd102} \lim_n &\sum\limits_{k\in I}\int_{\R^{d+m}}
a_k(\mx,\my)u_n(\mx,\my)\bar \rho_1(\my)\times
\\&\qquad
\times \overline{\Big({\cal A}_{\psi_{\Pd}}\circ
{\cal A}_{{k1}}\circ\dots \circ {\cal
A}_{{kd}}\circ {\cal
T}^{1-\sum\limits_{j=1}^d\frac{\alpha_{kj}}{\beta_j}}\Big)
(\ph_1 v_n)(\mx)} d\mx d \my=0, \nonumber
\end{align}
where  ${\cal A}_{{kj}}$,
is the multiplier operator with the symbol
$$
(1-\theta(\mxi)) \Big(-2\pi i \pi_{\Pd}(\mxi)\Big)^{\alpha_{kj}}_j=
(1-\theta(\mxi)) \frac{(-2\pi i\xi_j)^{\alpha_{kj}}}{|\mxi|_{\mbeta}^{\alpha_{kj}/\beta_j}}\,.
$$
Since the powers $\alpha_{kj}$, $j=1,\dots,d$, are
either grater than $d$ or natural numbers, the above symbol is the composition of the projection $\pi_{\Pd}$ and a smooth function (of class $\CC^d$), thus satisfying conditions of Theorem \ref{m1}. Thus, the corresponding
  operators ${\cal A}_{(-\pi_{\Pd})^{\alpha_{kj}}_j}$,
$j=1,\dots,d$ are $\LL p$ continuous and satisfy bound \eqref{marz_bnd}.

According to \eqref{sum1} and the definition of the main
symbol, we conclude that for every $k \notin I'$ it must be
$\sum\limits_{j=1}^d\frac{\alpha_{kj}}{\beta_j}<1$. By means of  Lemma
\ref{l1IIv} we conclude that for such an index $k$ the limit of the integral in \eqref{bd102} vanishes, and, due to the arbitrariness of test functions, the relation takes the form
\begin{equation}
\label{vv2} A B_l=0,
\end{equation} where $A$ is the principal symbol given by \eqref{glavni}.

%

According to Corollary \ref{korolar}, we can test \eqref{vv2} on the
function
$$
\phi(\mx,\my) \psi(\mxi)\frac{\overline{A(\mx,\my,\mxi)}}{|A(\mx,\my,\mxi)|^2+\delta},
$$ for an arbitrary $\phi\in \Cbc{\R^{d+m}}, \psi\in \pC d{\Pd}$. Thus, we obtain
\begin{align*}
&\Dupp{B_l}{\phi(\mx,\my)\psi(\mxi)\frac{|{A(\mx,\my,\mxi)}|^2}{|A(\mx,\my,\mxi)|^2+\delta}}=0\,, \nonumber
\end{align*}
and by letting $\delta \to 0$, using \eqref{ggnc} and the continuity of the functional $B_l$,
we conclude
$$
B_l = 0, \ \ \forall l\in \N.
$$
 From the definitions of the H-distributions and the truncation operator $T_l$,
 we conclude by taking in \eqref{rev1} test functions $\psi=1$ and $\phi_1\phi_2=\ph \times\rho$ for the previously chosen $\ph$ and $\rho$
 (see the beginning of the proof):
\begin{align}
\label{bd103} &0=\lim\limits_{n\to \infty} \int_{\Rd } \ph^2(x)
\Big|T_l \int_{\R^m}u_n(\mx,\my)\rho(\my)d\my\Big|^2 d\mx, \quad
l\in \N.
\end{align}
Now, using Lemma \ref{DHM}, we conclude that
$$
\big(\int_{\R^m} \rho(\my) u_n(\cdot,\my)d\my \big) \ \
\text{is strongly precompact in $\Ljl\Rd$}.
$$
\end{proof}

\section{Optimal velocity averaging in ${\rm L}^p$, $p\geq 2$, framework}

Using the method from Theorem \ref{t-main} we are able to
optimize the velocity averaging results when the sequence of
solutions to \eqref{general} are bounded in $\pL p{\R^{d+m}}$ for some
$p\geq 2$, under the classical non-degeneracy conditions given by
\eqref{kingnl}. We shall need the extension of the H-measures
introduced in \cite{LM2} whose existence and properties are restated
in the next theorem.
\begin{theorem}
\label{lfeb518-r}
Assume that a sequence $(u_n)$
 converges weakly to zero in $\Ld{\R^{d+m}}$ $\cap$ $ \Ld{\Rm; \pL p{\Rd}}$, $p\geq2$. Then the\-re exists a  measure $\mu \in \Ldws{\R^{2m}\!;
{\cal M}_{b}(\Rd\times \Pd)}$ such that for all $\phi_1\in \Ld{\Rm;
\pL {\tilde s'}{\Rd}}$, $\frac{1}{\tilde{p}'}+\frac{2}{p}=1$ (with
$\Lb\Rd$ being replaced by $\Cnl\Rd$ if $p=2$), $\phi_2\in
\Ldc{\R^{m};\Cnl\Rd}$,  and $\psi\in \pC d{\Pd}$ it holds
\begin{equation*}
\begin{split}
\lim\limits_{n'}\int\limits_{\R^{2m}}\int\limits_{\Rd}
&(\phi_1 u_{n'})(\mx,
\my)
\,\Bigl(\overline{{\cal A}_{\psi_\Pd} \,\phi_2 u_{n'}(\cdot,
\tilde\my)}\Bigr)(\mx) d\mx d\my d\tilde\my\\
&=\int\limits_{\R^{2m}} \langle
\mu(\my,\tilde\my,\cdot,\cdot),{\phi}_1(\cdot,\my)\overline{\phi_2}
(\cdot,\tilde\my)\otimes\overline\psi\rangle d\my d\tilde\my\,,
\end{split}
\end{equation*}
where ${\cal A}_{\psi_\Pd}$ is the (Fourier) multiplier operator on
$\R^d$ associated to $\psi\circ \pi_\Pd$.

Furthermore, the operator $\mu$   has the
form
\begin{equation}
\label{repr_1} \mu(\my, \tilde\my,\mx,\mxi)=f(\my,
\tilde\my,\mx,\mxi)\nu(\mx,\mxi) d\my d\tilde\my,
\end{equation} where $\nu\in {\cal M}_b(\R^d\times \Pd)$ is a non-negative scalar
Radon measure whose $\Rd$ projection $\int_\Pd d\nu(\mx,\mxi)$  can be extended to a bounded functional on
$\pL{\tilde p'}\Rd$ in the case $p>2$, while $f$ is a function from
$\Ld{\R^{2m}; \Lj{\Rd\times \Pd:\nu}}$.
\end{theorem}

By
$\Ldws{\R^{2m}; {\cal M}_{b}(\Rd\times \Pd)}$ we have denoted the Banach space of
weakly $\ast$ measurable functions $\mu: \R^{2m} \to {\cal
M}_{b}(\Rd\times \Pd)$  such that  $\int_{\R^{2m}}
\Nor{\mu(\my, \tilde\my)}^2  d\my d\tilde\my <\infty$.

An H-measure defined above is an object associated to a single $\LL
2$ sequence. However, there are no obstacles to adjoin a similar
object to different sequences as in the case of the H-distributions
(Theorem \ref{prop_repr}). This can be done by forming a vector
sequence, and consider non-diagonal elements of corresponding
(matrix) H-measure. Another way is to joint two sequences in a
single one by means of a dummy variable, as it is done in the next
theorem.

\begin{theorem}
\label{th10} Let $(u_n)$ be a bounded sequence in $\pL
2{\R^{d+m}}\cap\pL 2{\R^{m}; \pL p{\Rd}}$, for some $p\geq 2$, and
let $(v_n)$ be a sequence weakly converging to zero in $\pL
2{\R^{d}} \cap{\rm L}^q (\R^{d})$ where $1/q+1/p<1$. Then, after
passing to a subsequence (not relabeled), there exists a  measure
$\mu \in \Ldws{\R^{m}; {\cal M}_{b}(\Rd\times \Pd)}$ such that for
all $\phi_1 \in {\rm L}^{2}(\R^{m}; \Cnl\Rd)$, $\phi_2 \in \Cnl\Rd$,
$\psi\in \pC d{\Pd}$, we have
\begin{equation}
\label{rev11}
  \langle \mu, \phi_1 \overline{\phi_2} \otimes \overline\psi \rangle=\lim\limits_{n\to
\infty}\int_{\R^{d+m+k}} \phi_1(\mx, \my)u_n(\mx,\my)
\overline{{\cal A}_{\psi_\Pd} \big(\phi_2 v_n\big)(\mx)} d\mx
d\my\,.
\end{equation}

Furthermore, the measure $\mu$ is of the form
\begin{equation}
\label{repr} \mu(\my,\mx,\mxi)=f(\my,\mx,\mxi) d\nu(\mx,\mxi) d\my,
\end{equation} where $\nu\in {\cal M}_b(\R^d\times \Pd)$ is a
non-negative, bounded, scalar Radon measure, while $ f\in\Ld{\R^{m};
\Lj{\Rd\times \Pd:\nu}}$. We call it the generalised H-measure
corresponding to (sub)se\-quen\-ces (of) $(u_n)$ and $(v_n)$.
\end{theorem}

\begin{proof} Denote by $u$ an $\LL 2$ weak limit of the sequence
$(u_n)$ along a (non-relabeled) subsequence. Fix an arbitrary
non-negative compactly supported $\rho\in \Cc{\R^m}$ with the total mass equal to one. Let
\begin{equation*}
W_n(\mx,\my,\lambda)=
\begin{cases}
(u_n-u)(\mx,\my), & \lambda\in \oi 01\\
\rho(\my)v_n(\mx), & \lambda\in \oi {-1}0\\
0, & {\rm else}.
\end{cases}
\end{equation*}
Clearly, we have that $W_n \rightharpoonup 0$ in $\Ld{\R^d\times
\R^m\times\R}$, and by Theorem \ref{lfeb518-r} it admits a measure
$\tilde{\mu}\in \Ldws{\R^{2(m+1)}; {\cal M}_b(\R^d\times \Pd)}$ such
that for any $\tilde{\phi}_1\!\in \!{\rm L}^{2}(\R^{m+1}; \Cnl\Rd)$,
$\tilde{\phi}_2\in \Ldc{\R^{m+1}; \Cnl\Rd}$, and $\psi \in \Cp\Pd$
it holds
\begin{equation}
\label{rev111}
  \langle \tilde{\mu}, \tilde{\phi}_1 \overline{\tilde{\phi}_2} \otimes \overline\psi \rangle=\lim\limits_{n\to
\infty}\int_{\R^{d+2(m+1)}} (\tilde\phi_1 W_n)(\mx,\my,\lambda)
\overline{{\cal A}_{\psi_\Pd} \big(\tilde\phi_2
W_n (\cdot,\tilde\my,\tilde\lambda)\big)(\mx)} d\mx d{\bf
w},
\end{equation} where ${\bf w}=(\my,\tilde{\my},\lambda,\tilde{\lambda})\in
\R^{2m+2}$.

According to the representation \eqref{repr_1}  the measure $\tilde{\mu}$ is of the form
\begin{equation*}
\tilde{\mu}=\tilde
f(\my,\tilde{\my},\lambda,\tilde{\lambda},\mx,\mxi)
d\nu(\mx,\mxi)d\my d\tilde\my d\lambda d\tilde{\lambda}, \ \ \my,
\tilde{\my} \in \R^m, \; \lambda,\tilde{\lambda} \in \R,
\end{equation*} where $\nu\in {\cal M}_b(\R^d\times \Pd)$ is a non-negative scalar
Radon measure, while $\tilde f$ is a function from
$\Ld{\R^{2(m+1)}; \Lj{\Rd\times \Pd:\nu}}$.

By taking in \eqref{rev111}
$\tilde{\phi}_1(\mx,\my,\lambda)=\phi_1(\mx,\my) \otimes
\theta_1(\lambda)$, and
$\tilde{\phi}_2(\mx,\tilde{\my},\tilde{\lambda})$ $= \phi_2(\mx)
\otimes \rho_2(\tilde \my)\otimes\theta_2(\tilde{\lambda})$, where
$\phi_1\in {\rm L}^{2}(\R^{m}; \Cnl\Rd)$ and $\phi_2\in \Cnl\Rd$ are
arbitrary test functions, while $\theta_{1}=\chi_{\zi 01},
\theta_{2}=\chi_{\zi {-1}0}$, and $\rho_2(\tilde \my)=1$ for
$\tilde{\my}\in {\rm supp}\rho$, we see that the measure
$$
d\mu(\my, \mx,\mxi)=\left(\int_{-1}^0\int_{0}^1\int_{\R^{m}}
\tilde f(\my,\tilde{\my},\lambda,\tilde{\lambda},\mx,\mxi)\,
d\tilde{\my} d\lambda d\tilde{\lambda} \right) d\nu(\mx,\mxi)d\my,
$$
satisfies \eqref{rev11}.
\end{proof}

\begin{remark}
The last theorem is stated for sequence of functions $u_n$ being in
$\LL 2$ space with respect to the velocity variable $\my$, as this
was the setting in which generalised H-measures have been defined in
\cite{LM2}. However, if in addition one assumes that $(u_n)$ is
bounded in $\pL p{\R^m;\pL s\Rd}$ for some $p\in \oi 1\infty$, then
a test function can be taken merely from  ${\rm L}^{p'}(\R^{m};
\Cnl\Rd)$.
\end{remark}

By using the above characterisation of H-measures we are able to improve the main result of the paper, namely Theorem\ref{t-main}, in the case $p\geq 2$ by assuming merely the classical non-degeneracy
condition \eqref{kingnl} instead of the restrictive one given by \eqref{rndc}. Note that due to the lower regularity assumptions on the  coefficients the following theorem also generalises the velocity averaging results provided in \cite{LM2}.
\begin{theorem}
 Assume that $u_n\dscon 0$  weakly in
${\rm L}^{p}(\R^{d+m})\cap {\rm L}^2(\R^{d+m})$, $p\geq 2$, where
$u_n$ represent weak solutions to \eqref{general} in the sense of
Definition \ref{weaksol} (with conditions {\bf a)} and {\bf b)} together with the homogeneity assumption \eqref{gen_homog} being
fulfilled). Furthermore, assume that the classical non-degeneracy
conditions \eqref{kingnl} are satisfied.


Then, for any $\rho\in \pLc{2}{\R^m}$,
$$
\int_{\R^m}u_n(\mx,\my)\rho(\my)d\my \str 0 \ \ \text{ strongly in
$\Ljl\Rd$}.
$$

\end{theorem}

\begin{proof}
We stick to the notations from Theorem \ref{t-main}.

Denote by $B_l$ the H-distribution corresponding to the sequences
$(u_n)$ and $(v^l_n)$.

Thus, from the proof of Theorem \ref{t-main}, we conclude that
$B_l$ satisfies localization principle given by \eqref{vv2}. We
wish to prove that from here, under condition \eqref{kingnl}, it
follows that $B_l \equiv 0$.

Since  the sequence $(u_n)$ is bounded in ${\rm L}^2(\R^{d+m})$
then, according to Theorem \ref{th10}, together with the sequence
$(v^l_n)$, it forms the H-measure which coincides with the
H-distribution $B_l$ at least on the space
$\Cc{\R^{d+m};\pC d\Pd}$ where therefore it admits the
representation given by \eqref{repr}. By using the density arguments and the continuity of the H-distribution on ${\rm
L}^{p'}(\R^m;{\rm L}^{\bar{p}'}(\R^d);\pC d\Pd)$ (provided by Corollary \ref{korolar}),  for an arbitrary test function $g \in {\rm
L}^{p'}(\R^m;L^{\bar{p}'}(\R^d);\pC d\Pd)$ it
holds

\begin{equation}
\label{repr_va} \langle B_l,g\rangle=\int_{\R^m} \int_{\R^d\times \Pd}
g(\my,\mx,\mxi) f_l(\my,\mx,\mxi)
d\nu_l(\mx,\mxi) d\my,
\end{equation} for some $\nu_l\in {\cal M}_b(\R^d\times \Pd)$ and $f_l\in\Ld{\R^{m};
\Lj{\Rd\times \Pd:\nu}}$.

Now, take an arbitrary $\delta>0$, and for a $\rho\in \Ldc{\R^m}$ and
$\phi \in \Cc{\R^{d};\pC d\Pd}$ consider the test function
\begin{equation*}
\frac{\rho(\my)\phi(\mx,\mxi)\overline{A(\mx,\mxi,\my)}}{|A(\mx,\mxi,\my)|^2
+\delta}.
\end{equation*}
The localisation principle \eqref{vv2} implies
\begin{align*}
&\Dupp{B_l}{\rho(\my)\phi(\mx,\mxi)\frac{|{A(\mx,\my,\mxi)}|^2}{|A(\mx,\my,\mxi)|^2+\delta}}=0\,, \nonumber
\end{align*}
 which by means of representation \eqref{repr} and Fubini's theorem
takes the form
\begin{equation}
\label{fin_4} \int_{\R^d\times \Pd}\int_{\R^{m}}
\frac{\rho(\my)\phi(\mx,\mxi)|A(\mx,\mxi,\my)|^2}{
|A(\mx,\mxi,\my)|^2 +\delta}f_l(\my,\mx,\mxi) d\my
d\nu_l(\mx,\mxi)=0.
\end{equation}
Let us denote
$$
I_\delta(\mx, \mxi)=  \int_{\R^{m}}
\rho(\my)\frac{|A(\mx,\mxi,\my)|^2}{ |A(\mx,\mxi,\my)|^2
+\delta}f_l(\my,\mx,\mxi)d\my\,.
$$
According to the non-degeneracy condition \eqref{kingnl}, we have
$$
I_\delta(\mx, \mxi)\to   \int_{\R^{m}}
\rho(\my)f_l(\my,\mx,\mxi)d\my,
$$
as $\delta\to 0$ for $\nu- {\rm a.e. }\, (\mx, \mxi) \in \R^d\times
\Pd$. By using the Lebesgue dominated convergence theorem, it
follows from \eqref{fin_4} after letting $\delta\to 0$:
\begin{align*}
&\Dup{B_l}{\rho\otimes\phi}= \int_{\R^d\times
\Pd}\int_{\R^{m}}\rho(\my)\phi(\mx,\mxi)f_l(\my,\mx,\mxi) d\my
d\nu_l(\mx,\mxi)
=0\,,\nonumber
\end{align*}
i.e. $B_l=0$ for every $l$.
Now, as in the proof of Theorem \ref{t-main}, we
conclude that the sequence of truncated averages $T_l \int_{\R^m}u_n(\mx,\my)\rho(\my)d\my$ is strongly precompact in $\Ljl\Rd$, which together with  Lemma
\ref{DHM}  concludes the theorem.

\end{proof}

\section{Appendix}

We are able to use previously introduced techniques to point out a
connection between the H-distribu\-tions that we have introduced and
the H-measures that we used in \cite{LM2} (Theorem \ref{lfeb518-r}).
We shall need a kind of truncation function again:
$$
u^l=
\begin{cases}
u, & l<|u|\leq l+1\\
0, & else.
\end{cases}
$$

The following theorem holds.

\begin{theorem}
\label{prop_repr1} Let $(u_n)$ be a sequence bounded in $\pL
p{\R^{d+m}}$, $p>1$. Let $(v_n)$ be a sequence weakly converging to
zero in ${\rm L}^2 (\R^{d})\cap {\rm L}^s (\R^{d})$ for every finite
$s\geq p'$. Then, the following representation holds for the
H-distribution $B$ corresponding to  (sub)sequences (of) $(u_n)$ and
$(v_n)$
\begin{equation*}
B= \sum\limits_{l=1}^\infty f^l(\my,\mx,\mxi)
d\nu^l(\mx,\mxi) d\my,
\end{equation*}
where $f^l(\my,\mx,\mxi) d\nu^l(\mx,\mxi) d\my $ are
 generalised H-measures corresponding to  $(u_n^l)$ and
 $(v_n)$, $l\in \N$.

\end{theorem}

\begin{proof}
First remark that we can write
\begin{equation*}
u_n=\sum\limits_{l=1}^\infty u_n^l,
\end{equation*}
and that $(u_n^l)$ is a sequence of functions with $\LLb$ norms
bounded by $l$. Denote by $(u_n)$ a non-relabeled subsequence of
$(u_n)$ such that for each $l$ $(u_n^l)$ converges weakly $\ast$ in
$\Lb{\R^d\times \R^m}$ toward a limit $u^l$.

By means of the last relation the limit from \eqref{rev1} can be express as
\begin{equation}
\label{pom}
\lim\limits_{n\to \infty}\int_{\R^{d+m}} \phi_1(\mx,
\my)\sum\limits_{l=1}^\infty u_n^l(\mx,\my) {\cal
A}_{\psi_\Pd}( \phi_2 v_n)(\mx) d\mx d\my\,.
\end{equation}
The test function $\phi_1$ is taken from the space $\pL {{p}'}{\R^{m}; \pL{\bar{p}'}\Rd}
$ for some $\bar p \in \oi 1p$, and the integral is well defined since by assumption $(v_n)$ is specially bounded in ${\rm L}^q (\R^{d})$, with $q$ given in a).

The result of the theorem will follow easily if we show the summation sign in \eqref{pom} can be put in front of the limit.

To this effect notice that $\sum\limits_{l=L}^\infty u^l_n$ is supported within the set
$$
\Omega_n^L=\{ \mx\in\Omega:\, u_{n}(\mx) > l \}
$$
for which we have shown in Lemma \ref{trunc-l} that
\begin{equation}
\label{conv0} \lim\limits_{L\to \infty} \sup\limits_{n\in \N} {\rm
meas}(\Omega_n^L) =0.
\end{equation}

For some  $\phi_1^\eps\in\Cbc{\R^{d+m}}$ such that $\nor{\phi_1-\phi_1^\eps}{\pL{p'}{\R^m; \pL{\bar p'}\Rd}} < \eps$ we have

\begin{align*}
\Big|\lim\limits_{n\to \infty}
&\int_{\R^{d+m}}
\phi_1(\mx, \my)\sum\limits_{l=L}^\infty u_n^l(\mx,\my)
{\cal A}_{\psi_\Pd}( \phi_2 v_n)(\mx) d\mx d\my\,\Big|\\
&=\Big|\lim\limits_{n\to \infty}\int_{\Omega_n^L} \phi_1(\mx, \my)u_n(\mx,\my) {\cal A}_{\psi_\Pd}( \phi_2 v_n)(\mx) d\mx
d\my \,\Big|\nonumber\\
&\leq C
\limsup_n \|u_n\|_{{\rm L}^{{p}}(\Omega_n^L)} \|\phi_2 v_n\|_{\pL{q}{\Rd}}
\Big(\eps + \|\phi_1^\eps\|_{{\rm L}^{\bar p'}(\Omega_n^L)}\Big)
\nonumber,
\end{align*}
where we have used that $\pL {\bar p'}K \hookrightarrow \pL { p'}K$ for a compact set $K\subset\Rd$. As $\eps$ is arbitrary, by using \eqref{conv0} it follows that the   above expression goes to zero as $L$ goes to infinity.
Thus we can shift the summation sign in \eqref{pom}, and we get
\begin{align*}
\langle B, \phi_1\overline{\phi_2} \otimes\overline\psi\rangle&=\sum\limits_{l=1}^\infty \lim\limits_{n\to \infty}\int_{\R^{d+m}}
\phi_1(\mx, \my) u_n^l(\mx,\my) {\cal A}_{\psi_\Pd}(
\phi_2 v_n)(\mx) d\mx d\my\nonumber.
\end{align*}
Now it is enough to rely on \eqref{rev11} to conclude the
proof.
\end{proof}

If the sequence $(u_n)$ from Theorem \ref{prop_repr} is bounded in
$\LLd$ with respect to the velocity variable, then the corresponding
H-distribution can be represented as an infinite (weighted) sum of
the H-distributions $\mu_i$, $i\in \N$, corresponding to the
sequences $(\int_{\R^m} u_n(\cdot,\my)e_i(\my)d\my)$ and
$(v_n)$, where $\{e_i\}_{i\in \N}$ is an orthonormal basis in
$\Ld{\R^m}$. A similar representation holds for the H-measures
(see the proof of \cite[Proposition 12]{LM2}), but in that case, by using the positivity property, it can be  further simplified to the form given in  \eqref{repr_1}.

\begin{proposition}
\label{prop_reprAp} \hskip -1mm Denote by $\mu$ the generalised H-distribution
corresponding to  (sub)\-se\-quen\-ces (of) $(u_n)$, taken to be
bounded in $\pL{p}{\Rd; \pL{2}{\R^m}}\cap {\rm L}^p(\R^{d+m})$, $p\in \ozi 12$, and $(v_n)$, weakly converging to
zero in $\Ld\Rd \cap {\rm L}^{q}(\R^d)$, for some  $q\geq 2$.
Denoting by  $\mu_{i}$ H-distributions corresponding  to $\big(\int_{\R^m}
u_n(\cdot,\my)e_i(\my)d\my\big)$ and $(v_n)$, the following
representation holds
\begin{equation}
\label{repr2} \langle \mu, \phi_1  \overline{\phi_2} \otimes \overline{\psi}
\rangle=\sum\limits_{i=1}^\infty \langle \mu_{i},
 \int _{\R^m}\phi_1(\cdot,\my) e_i(\my) d\my\,
\overline{\phi_2} \otimes \overline{
\psi}
\rangle,
\end{equation}
with test functions
 $\phi_{1}\in {\rm L}^{p'}(\R^m;{\rm
L}^{\bar{p}'}(\R^d))\cap \pL{\bar{p}'}{\Rd; \pL{2}{\R^m}}$,  $\phi_2\in \Cnl\Rd$,
and $\psi \in \pC{d}\Pd$.

\end{proposition}

%

\begin{proof}

Rewrite an arbitrary test function $\phi_{1}$ as
\begin{equation*}
\phi_1(\mx, \my)=\sum\limits_{i=1}^\infty c_{i}(\mx)
 e_i(\my),
\end{equation*}
where $c_{i}(\mx)=\int_{\R^m}\phi_1(\mx,\my)e_i(\my)d\my$ and $\left(\sum_i |c_i(\mx)|^2\right)^{{1/ 2}}$ belongs to ${\rm L}^{\bar{p}'}(\R^d)$.

According to Theorem \ref{prop_repr}, for $\phi_1$ from above, and $\phi_2\in \Cnl\Rd$, $\psi \in \pC{d}\Pd$ we have that
\begin{align}
\label{last}
\langle \mu, \phi_1 \overline{\phi_2} \otimes \overline{ \psi}
\rangle
&=\lim\limits_{n\to \infty}\int_{\R^{d+m}}(\phi_1
u_n)(\mx,\my)
\,\overline{{\cal A}_{\psi}(\phi_2 v_n)(\mx)} d\mx d\my\nonumber\\
&=\lim\limits_{n\to \infty}\int_{\R^d}\sum\limits_{i=1}^\infty
c_{i}(\mx) \int_{\R^m} u_n(\mx,\my)e_i(\my)d\my\,
\overline{{\cal A}_{\psi}\,(\phi_2 v_n)(\mx)}
 \,
d\mx\,,
\end{align}
By  taking into account properties of the coefficients $c_i$,  a procedure   similar to the one applied in the preceding theorem enables us to estimate the limit of
$$
\int_{\R^d}
\sum\limits_{i=L}^\infty c_{i}(\mx)
\int_{\R^m} u_n(\mx,\my)e_i(\my)d\my\,
\overline{{\cal A}_{\psi}\,(\phi_2 v_n)(\mx)} \,,
$$
which goes to zero as $L$ approaches infinity, uniformly with respect to $n$.
Thus we can relocate the summation sign in \eqref{last} in order to get
\begin{align*}
\langle \mu, \phi_1 \overline{\phi_2} \otimes \overline{ \psi}
\rangle
&=\sum\limits_{i=1}^\infty  \langle \mu_{i}(\mx,\mxi), c_i
\overline{\phi_2} \otimes \overline{ \psi}\rangle
=\sum\limits_{i=1}^\infty
\langle \mu_{i}(\mx,\mxi), \int_{\R^{m}}\!\!\!\phi_1(\cdot,\my) e_i(\my)d\my\, \overline{\phi_2} \otimes \overline{ \psi}
 \rangle,
\end{align*}
which completes the proof of \eqref{repr2}.
\end{proof}

\section*{Acknowledgements}

This paper has partially been developed while Martin Lazar was a
Postdoctoral Fellow on the Basque Center for Applied Mathematics
(Bilbao, Spain) within the  NUMERIWAVES FP7-246776 project, and
while Darko Mitrovic was part time postdoc at the University of
Bergen financed by the Research Council of Norway. The work is also
supported in part by the bilateral Croatian--Montenegro project {\it
Transport in highly heterogeneous media}, as well as by the DAAD
project {\it Center of Excellence for Applications of Mathematics}.


\begin{thebibliography}{99}

\bibitem{Ago} {\sc V.\,I.\,Agoshkov}, {\em Spaces of functions with differential-difference
characteristics and smoothness of solutions of the transport
equation}, Soviet Math. Dokl. {\bf 29} (1984), 662--666.

\bibitem{AM} {\sc N.\,Antoni\'c, D.\,Mitrovi\'c}, {\em H-distributions -- an extension of the H-measures in $L^p-L^q$ setting}, Abstr. Appl. Anal.
{\bf 2011} (2011), 12 pp.

\bibitem{AL} {\sc N.\,Antoni\'c, M.\,Lazar}, {\em Parabolic H-measures}, J. Funct. Anal. {\bf  265} (2013), 1190--1239.

\bibitem{BWM} {\sc D.~Benson, S.~Wheatcraft, M.~Meerschaert}, {\em The
fractional-order governing equation of L\'evy motion},  Water
Resources Res. {\bf 36} (2000), 1413--1423.

\bibitem{BWM1} {\sc D.~Benson, R.~Schumer, S.~Wheatcraft, M.~Meerschaert},
{\em Fractional dispersion, L\'evy motion, and the MADE tracer tests},
Transport Porous Media {\bf 42} (2001), 211--240.

\bibitem{dentz} {\sc B.~Berkowitz, A.~Cortis, M.~Dentz, H.~Scher}, {\em
Modeling non-Fickian transport in geological formations as a
continuous time random walk}, Reviews of Geophysics {\bf 44} (2009),
1--49.

\bibitem{CJ} {\sc S.~Cifani, E.~R.~Jakobsen}, {\em
    Entropy solution theory for fractional degenerate convection-diffusion
    equations},     Ann. Inst. H. Poincare Anal. Non Lineaire {\bf 28} (2011), 413--441.

\bibitem{CT} {\sc R.~Cont, P.~Tankov}, {\em Financial modelling with jump processes},
Chapman \&  Hall/CRC Financial Mathematics Series,  2004.

\bibitem{13} {\sc R.\,J.\,DiPerna, P.\,L.\,Lions, Y.\,Meyer}, {\em $L^p$ regularity of velocity
averages}, Ann. Inst. H. Poincar\' e Anal. Non Lin\' eaire {\bf 8}
(1991), 271--287.

\bibitem{8} {\sc R.\,J.\,Diperna, P.\,L.\,Lions},
{\em Global Solutions of Boltzmann Equations and the Entropy
Inequality}, Arch. Rat. Mech. Anal. {\bf 114} (1991), 47--55.

\bibitem{DHM} {\sc G.~Dolzmann,~N.~Hungerbuhler,~S.~M\"uller},
{\em Nonlinear elliptic systems with measure valued right-hand
side}, Math. Zeitschrift 226, (1997) 545--574.

\bibitem{Ger}
{\sc P.\,G\' erard}, {\em Microlocal Defect Measures}, Comm. Partial
Differential Equations  {\bf 16} (1991), 1761--1794.

\bibitem{18} {\sc F.\,Golse, L.\,Saint-Raymond}, {\em Velocity
averaging in $L^1$ for the transport equation},
C. R. Acad. Sci. Paris S\'er. I Math.
{\bf 334} (2002), 557--562.

\bibitem{Gra} {\sc L.~Grafakos}, {\em Classical Fourier Analysis},
Graduate Text in Mathematics 249,  Springer Science and Business
Media, 2008

\bibitem{LM2} {\sc M.\,Lazar, D.\,Mitrovi\'c}, {\em Velocity averaging -- a general framework},
Dynamics of PDE {\bf 3} (2012), 239--260.

\bibitem{LM_crass} {\sc M.\,Lazar, D.\,Mitrovi\'c}, {\em On an extension of a bilinear functional on $L^p(\R^d)\times E$
to Bochner spaces with an application on velocity averaging},
C. R. Acad. Sci. Paris S\'er. I Math.  {\bf 351} (2013), 261--264.

\bibitem{Liocc} {\sc P.\,L.\,Lions},
{\em A concentration compactness principle in the calculus of
variations. The limit case, parts 1 and 2}, Rev. Mat. Iberoamericana
{\bf 1} (1985), No. 1, 145--201, No. 2, 45--121.

\bibitem{LPT} {\sc P.\,L.\,Lions, B.\,Perthame, E.\,Tadmor}, {\em A kinetic formulation of multidimensional scalar
conservation law and related equations}, J. Amer. Math. Soc. {\bf 7}
(1994), 169--191.

\bibitem{MI} {\sc D.~Mitrovi\'c, I.~Ivec}, {\em A generalization of $H$-measures and application on purely fractional scalar conservation laws},
Communication on Pure and Applied Analysis {\bf 10} (2011),
1617--1627.

\bibitem{pan_arma} {\sc E.~Yu.~Panov}, {\em Existence and strong pre-compactness properties for
entropy solutions of a first-order quasilinear equation with
discontinuous flux}, Archives Rat. Mech. Anal. {\bf 195} (2010)
643--673.

\bibitem{32} {\sc B.\,Perthame, P.\,Souganidis}, {\em A limiting case for velocity averaging}, Ann.
Sci. Ec. Norm. Sup. {\bf 4} (1998), 591--598.

\bibitem{Sw}
{\sc Laurent Schwartz}, {\em Th\'eorie des distributions}, Hermann,
Paris 1978.

\bibitem{36} {\sc T.\,Tao, E.\,Tadmor}, {\em Velocity Averaging, Kinetic Formulations,
and Regularizing Effects in Quasi-Linear Partial Differential
Equations}, Comm. Pure Appl. Math. {\bf 60} (2007), 1488--1521.

\bibitem{Tar} {\sc L.\,Tartar}, {\em H-measures, a new approach for studying homogenisation,
oscillation and concentration effects in PDEs}, Proc. Roy. Soc.
Edinburgh. Sect. A {\bf 115} (1990), 193--230.

\bibitem{tar_book} {\sc L.\,Tartar}, {\em The General Theory of Homogenization: A Personalized Introduction},
Springer-Verlag Berlin Heidelberg, 2009.



\end{thebibliography}
\end{document}